\documentclass[12pt,reqno]{amsart} 

\usepackage{amsmath,amssymb,amsthm,amsfonts}
\usepackage{color}
\usepackage{graphicx,subcaption}
\usepackage{hyperref}
\usepackage{epstopdf}
\usepackage{bm}
\usepackage{enumerate}
\usepackage{amsthm}
\usepackage[margin=4cm]{geometry}

	\makeatletter
	\def\blfootnote{\gdef\@thefnmark{}\@footnotetext}
	\makeatother

\usepackage{todonotes}

\theoremstyle{plain} 
\newtheorem{thm}{Theorem}[section]
\newtheorem{lem}[thm]{Lemma}
\newtheorem{prop}[thm]{Proposition}

\newtheorem{conj}[thm]{Conjecture}

\theoremstyle{definition}
\newtheorem{defn}[thm]{Definition}
\newtheorem{ques}[thm]{Question}

\theoremstyle{remark}
\newtheorem{rem}[thm]{Remark}

\numberwithin{equation}{section}

\def\RR{\mathbb{R}}

\newcommand{\relper}{\mathrm{Rel}\;\mathrm{Per}}
\newcommand{\firstvar}{\mathrm{First}\;\mathrm{Var}}
\DeclareMathOperator{\vol}{Vol}
\DeclareMathOperator{\area}{Area}

\begin{document}

\title{On the relative isoperimetric problem for the cube}

\author{Gregory R. Chambers and Lawrence Mouill\'{e}}
\address{Department of Mathematics, Rice University, Houston, TX, 77005}
\email{gchambers@rice.edu}
\address{Department of Mathematics, Trinity University, San Antonio, TX, 78212}
\email{lmouille@trinity.edu}
\date{\today}

\begin{abstract}
    In this article, we solve the relative isoperimetric problem in $[0,1]^3$ for orthogonal polyhedra.  
    Up to isometries of the cube or sets of measure $0$,
    the minimizers are of the form $[0,\epsilon]^3$, $[0,\epsilon]^2 \times [0,1]$, or $[0,\epsilon] \times [0,1]^2$ for some $\epsilon > 0$.
    
    This should be compared to the conjectured minimizers for the unconstrained relative isoperimetric problem in $[0,1]^3$, which are (up to
    isometries and sets of measure $0$) of the form $\left( B^3(\epsilon) \right) \cap [0,1]^3$, $\left( B^2(\epsilon) \times [0,1] \right) \cap [0,1]^3$,
    or $[0,\epsilon] \times [0,1]^2$ for some $\epsilon > 0$.  
    Here, $B^k(\epsilon)$ is the closed ball in $\mathbb{R}^k$ of radius $\epsilon$
    centered at the origin.
\end{abstract}

\maketitle

\blfootnote{The first author was supported by NSF Grant DMS-1906543, and the second author was supported by NSF Award DMS-2202826.}

\section{Introduction}
\label{sec:intro}

Let $[0,1]^n$ be the unit cube in $\mathbb{R}^n$ with the Euclidean metric.  
If $U \subset [0,1]^n$ is a measurable set of locally finite perimeter, let $\vol(U)$ be the $n$-dimensional Lebesgue measure of $U$, and let $\textrm{Per}(U)$ be the perimeter of $U$ (i.e. the $(n-1)$-dimensional Hausdorff measure of the essential boundary of $U$).  
We define the relative perimeter of $U$, denoted by $\relper(U)$, as the $(n-1)$-dimensional Hausdorff measure of the essential boundary of $U$ \emph{except} for the portion of that boundary which lies on $\partial [0,1]^n$.  
The relative isoperimetric problem concerns the following question:

\begin{ques}
    \label{ques:isoperimetric}
    Fixing $V \in [0,\frac{1}{2}]$, what is
    \[
        I(V) = \inf \{\relper(U) : U \text{ is as above}, \vol(U) = V \},
    \]
    and what are the minimizers?
\end{ques}

We note that one could consider Question~\ref*{ques:isoperimetric} for $V \in [0,1]$, but the minimizers for $V\in[\frac{1}{2},1]$ are simply the complements in $[0,1]^n$ of the minimizers for $V\in[0,\frac{1}{2}]$.
The function $I:[0,1] \to \mathbb{R}$ is called the \textit{isoperimetric profile} of $[0,1]^n$.
The answer to Question~\ref*{ques:isoperimetric} is conjectured the be the following:

\begin{conj}
    \label{conj:isoperimetric}
    The minimizers for Question~\ref*{ques:isoperimetric}, up to isometries of $[0,1]^n$ and sets of measure $0$, are of the form $(B^m(r) \cap [0,1]^m) \times [0,1]^{n-m}$ for some $r \geq 0$ and $m \in \{1,\dots,n\}$.
    Here, $B^m(r)$ is the closed ball of radius $r$ in $\mathbb{R}^m$ centered at the origin.  
\end{conj}

\begin{rem}\label{rem:dim2}
    In dimension $n=2$, Conjecture \ref*{conj:isoperimetric} is true.  
    One can prove it by observing that any minimizer (via the first variation formula) must have boundary equal to an arc of a circle or a straight line segment, and the boundary must meet the boundary of the square $[0,1]^2$ perpendicularly.  
    Combining this with the orthogonal Steiner symmetrizations (discussed in Section \ref{sec:reduction}), one obtains the desired result.  
\end{rem}

Conjecture~\ref*{conj:isoperimetric} is still open in dimensions $\geq 3$.    
Proving that minimizers exist follows from a standard
compactness argument (since $[0,1]^n$ is compact).  
In addition, we obtain that the boundary of a minimizer must be regular except for a set of dimension
at most $n - 8$, and must meet $\partial [0,1]^n$ perpendicularly.
In dimension $3$, Conjecture~\ref*{conj:isoperimetric} is called the spheres-tubes-slabs conjecture, due to the nature
of the three possible types of minimizers described in Conjecture~\ref*{conj:isoperimetric}; see Figure~\ref*{fig:sphere-tube-slab}.

\begin{figure}[h]
    \includegraphics[width=\textwidth]{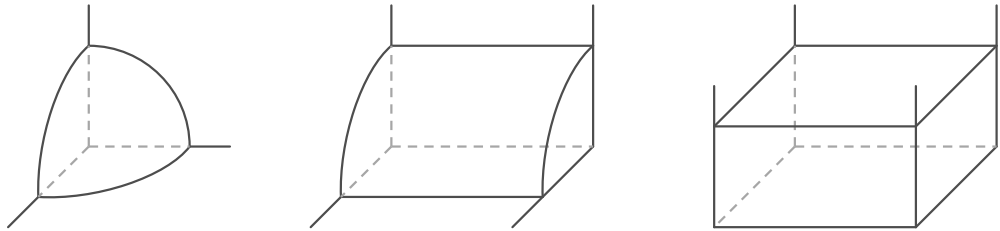}
    \caption{The proposed minimizers of relative perimeter in the spheres-tubes-slabs conjecture.}
    \label{fig:sphere-tube-slab}
\end{figure}


The spheres-tubes-slabs conjecture has been studied extensively, for example in \cite{Rit} and \cite{RitRos}; \cite{Ros2020THEIP} also has a nice description of it
 and related results.  
In \cite{Rit}, Ritor\'{e} proved that there are subsets of $[0,1]^3$ which have zero first variation,
are stable (with respect to the second variation), and whose boundaries are not of the form described in Conjecture \ref*{conj:isoperimetric} (see \cite{Lawson1970}, \cite{Ross1992},  and pages 10 and 18 of \cite{Ros2020THEIP}, 
).  
This result shows that one cannot hope to resolve the spheres-tubes-slabs conjecture simply by studying the first and second variations.


One of the main obstructions to addressing Question \ref*{ques:isoperimetric} is that the unit cube has
only a finite number of symmetries.  
In $\mathbb{R}^n$, we have all translation and reflection symmetries, and this allows us
to use the Steiner symmetrization process to prove that balls minimize perimeter (see the description of Steiner symmetrizations in Section \ref{sec:reduction}).

One special case which has been resolved in all dimensions is the case where the volume $V=\frac{1}{2}$.  In this case it is known
that the minimizers are of the form $[0,1/2] \times [0,1]^{n-1}$ (up to isometries of $[0,1]^n$); see, for example, \cite{Ledoux}.


The purpose of this article is to consider an analogous question to Question~\ref*{ques:isoperimetric}.  
We will first need the following definition:

\begin{defn}
    \label{defn:cubical_subset}
    We say that a compact subset $X \subset [0,1]^n$ is \emph{cubical} if its boundary 
    is contained in the union of finitely many hyperplanes, all of whose outer normal vectors lie in the set $\{\pm e_1, \pm e_2, \dots , \pm e_n \}$.
    Here, $e_1,\dots,e_n$ is the standard basis for $\RR^n$.
\end{defn}

\begin{rem}
    The term ``orthogonal polyhedra" has also been used to refer to these objects in the literature.  
    For the $2$-dimensional case, they are called ``rectilinear polygons."
\end{rem}


We now consider Question~\ref*{ques:isoperimetric} restricted to these cubical sets, and we make the following conjecture, analogous to
Conjecture~\ref*{conj:isoperimetric}:

\begin{conj}
    \label{conj:cubical_isoperimetric}
    Up to isometries of $[0,1]^n$ and sets of measure $0$, the minimizers for Question~\ref*{ques:isoperimetric} restricted to cubical subsets of $[0,1]^n$
    are of the form $[0,a]^m \times [0,1]^{n-m}$ for some $a \in [0,\frac{1}{2}]$ and $m \in \{1,\dots,n\}$.
\end{conj}

Our main theorem shows that Conjecture~\ref*{conj:cubical_isoperimetric} holds in dimension $n = 3$:

\begin{thm}
    \label{thm:main}
    Fix a volume $V \in [0,\frac{1}{2}]$.  
    Then the cubical subsets of $[0,1]^3$ which have measure $V$
    and minimal relative perimeter among all cubical subsets are of the form $[0,a]^m \times [0,1]^{3-m}$ for some $a \in [0,\frac{1}{2}]$ and $m \in\{1,2,3\}$, up to isometries
    of $[0,1]^3$ and sets of measure $0$.  
\end{thm}

\begin{figure}[h]
    \includegraphics[width=\textwidth]{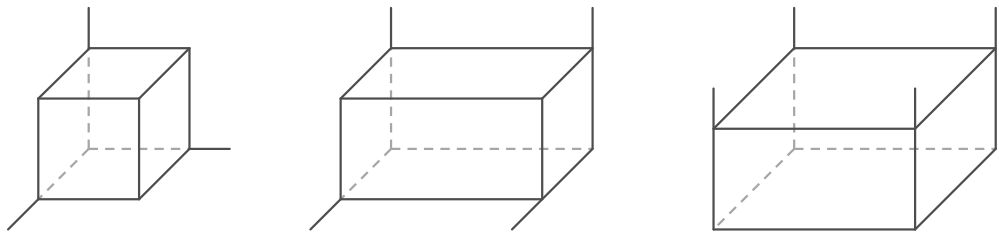}
    \caption{The minimizers of relative perimeter among cubical subsets.}
\end{figure}

For context, Loomis and Whitney proved in \cite{LM} that the product of the $(n-1)$-dimensional areas of the orthogonal projections of any measurable subset of $\mathbb{R}^n$ is at least the $n$-dimensional measure of that set to the power $n-1$.
The Loomis-Whitney inequality has been extended in a variety of ways, most notably to the
Brascamp-Lieb-Luttinger inequality (see \cite{BLL}).
We can view Theorem~\ref*{thm:main} as a relative version of the Loomis-Whitney inequality, in that it also produces an upper bound on the square of the volume as a function of
the product of the \emph{relative} projections of the set into each orthogonal direction.

The Loomis-Whitney inequality can also be used to prove the isoperimetric inequality in $\mathbb{R}^n$ with a nonsharp
constant.  
Similarly, Theorem \ref*{thm:main} can be used to give a relative isoperimetric inequality with a nonsharp constant.

Finally, if we are working with $\ell^1$ perimeter (instead of the standard $\ell^2$ perimeter), Theorem~\ref*{thm:main}
does answer the question completely.  
This is because in $\ell^1$ any set can be approximated by
cubical sets both in volume and perimeter; this is not possible in $\ell^2$.  
Note that here we are
using the $\ell^1$ (respectively $\ell^2$) norm to define the corresponding Hausdorff measures on subsets
of $[0,1]^n \subset \mathbb{R}^n$.
There has been interest in the related \textit{``double bubble''} problem for the $\ell^1$ norm, with the $\RR^2$ case being solved in \cite{Wulff}, and a conjectured solution for $\RR^3$ given in \cite{crystal}.

The proof of Theorem~\ref*{thm:main} will have two components.  
The first component is to reduce the problem to what we call \textit{special} cubical subsets.
Since the definition of \textit{special} is somewhat technical, we postpone
its statement until Section~\ref{sec:reduction}.
The main proposition of Section~\ref*{sec:reduction} is the following:
\begin{prop}
    \label{prop:reduction}
    Suppose that $X$ is a cubical subset of $[0,1]^n$.  
    Then there is a special cubical subset $X^*$ with $\vol(X^*) = \vol(X)$ such that
    \[
        \relper(X^*) \leq \relper(X).
    \]
\end{prop}

The proof of this proposition first involves applying three orthogonal Steiner symmetrizations to the set; this operation
was developed by Steiner in the 1840s to prove that, among convex sets, the ball minimizes perimeter.  
The remainder
of the proof of Proposition~\ref*{prop:reduction} will involve studying certain variations for cubical subsets and arguing via first variation methods.  
We note that the results of Section \ref*{sec:reduction} apply to all dimensions.

The remainder of the proof of Theorem~\ref*{thm:main} is contained in Sections \ref*{sec:identify} and \ref*{sec:unique}.
In Section \ref{sec:identify}, we show that the only special cubical subsets which minimize relative perimeter in dimension $3$ are the ones described in Theorem~\ref*{thm:main}.  
This involves using the first variations studied in Section~\ref*{sec:reduction} combined with a case-by-case examination of possible special cubical subsets.
In Section \ref{sec:unique}, we show that the minimizers identified in Section \ref*{sec:identify} are indeed the \emph{unique} minimizers of relative perimeter among all cubical subsets.

\section{Reduction}\label{sec:reduction}

The purpose of this section is to prove Proposition~\ref*{prop:reduction}.
To begin, we will define a Steiner symmetrization, a process of modifying a set which maintains volume and does not increase perimeter.
This type of process was first introduced in the 1840s by Steiner in \cite{steiner} to prove that convex subsets of $\mathbb{R}^2$ which minimize perimeter are balls (see also \cite{BurchardCourse}). 
We will define Steiner symmetrizations for cubical subsets of $[0,1]^n$ here, but first, we set up some notation and terminology that will be used throughout this section.

\begin{defn}
    Let $e_1,\dots,e_n$ denote the standard basis for $\RR^n$.
    For every $s \in [0,1]$ and $i \in \{1, \dots, n\}$, we define $H_s^i$ to be the \textit{hyperplane} $\{ \sum x_j e_j \in \mathbb{R}^n : x_i = s \}$.
    Given a cubical subset $X \subset [0,1]^n$, we define the \emph{slice} of $X$ with respect to $i$ and $s$ to be $X_s^i = \partial X \cap H_s^i$.
    We say that $s$ is a \emph{singular point} of $X$  with respect to $e_i$ (or simply $i$) if the slice $X_{s}^i$ has Hausdorff dimension $n - 1$.
\end{defn}


\begin{rem}
    Notice that for any cubical subset $X$, because $\partial X$ is contained in finitely many hyperplanes, it follows that $X$ has finitely many critical points with respect to any direction $e_i$.
\end{rem}


Now we define the Steiner symmetrization of cubical subsets.

\begin{defn}
    \label{defn:Steiner}
    Let $X \subseteq [0,1]^n$ be a cubical subset.
    Given $i \in \{1,\dots,n\}$ and $y \in H_0^i \cap [0,1]^n$, let $\ell$ denote the line with direction $e_i$ that is incident with $y$, and define $f(y)$ to be the measure of $\ell \cap X$.
    We define the interval
    \[
        I(y) = \begin{cases}
            [0, f(y)] & \text{if } f(y) \neq 0,\\
            \varnothing   & \text{if } f(y) = 0.
        \end{cases}
    \]
    The \emph{Steiner symmetrization} of $X$ in the direction $e_i$ is then defined as 
    \[
        S_{e_i}(X) = \bigcup_{y \in H_0^i \cap [0,1]^n} \{ y + t e_i :  t \in I(y) \}.
    \]
    
    
    
\end{defn}



We will now detail several properties of Steiner symmetrizations for cubical regions, including the fact that they preserve volume without increasing relative perimeter.


\begin{prop}
    \label{prop:steiner}
    Let $X$ be a cubical subset of $[0,1]^n$ and let $i\in\{1,\dots,n\}$.
    Then following hold:
    \begin{enumerate}
        \item $S_{e_i}(X)$ is also a cubical subset.  \label{item:symm_is_cubical}
        \item $\vol(X) = \vol(S_{e_i}(X))$. \label{item:vol_pres}
        \item $\relper(X) \geq \relper(S_{e_i}(X))$. \label{item:relper_dec}
        \item $\relper(X) = \relper(S_{e_i}(X))$ if and only if $S_{e_i}(X) = X$ up to the isometries of $[0,1]^n$ and sets of measure zero. \label{item:relper_same}
        \item If $S_{e_i}(X) = X$, then $S_{e_i}(S_{e_j}(X))=S_{e_j}(X)$ for all $j\in\{1,\dots,n\}$. \label{item:succ_sym}
    \end{enumerate}
\end{prop}

\begin{proof}
    To prove (\ref*{item:symm_is_cubical}), without loss of generality, we may assume that $i = n$.
    Let the singular points of $X$ in the direction $e_n$ be $s_1 < s_2 < \dots < s_{m-1} < s_m$.
    We will show that $\partial S_{e_n}(X)$ is contained in a union of finitely many hyperplanes $H_r^j$.
    Considering any point $x\in\partial S_{e_n}(X)$, if there is some $j \in \{1 ,\dots, n-1\}$ and some singular point $t$ with respect to $e_j$ such that $x$ lies in
    the hyperplane $H_t^j$, then there are only finitely many such hyperplanes.
    Alternatively, $x$ could be contained in one of finitely many coordinate hyperplanes $H_0^j$.
    For all other points $x \in \partial S_{e_n}(X)$, we have that $x \in H_{f(\hat{x})}^n$ by definition of $S_{e_n}(X)$, where $\hat{x}$ denotes the image of $x$ under the orthogonal projection onto the coordinate hyperplane $H_0^n$.
    We will show that there are finitely many possible values of $f(\hat{x})$ as $x$ varies.

    Considering any line $\ell$ with direction $e_n$ passing through a point $y \in H_0^n \cap [0,1]^n$, the intersection $\ell \cap X$ is either empty or consists of a finite number of disjoint intervals with endpoints whose $n$th coordinates lie in $\{s_1, \dots, s_m\}$.
    In particular, there is an ordered set of indices $1 \leq k_1 < k_2 < \dots < k_{2p-1} < k_{2p} \leq m$, dependent on $\ell$, such that the $n$th components of all elements in $\ell \cap X$ comprise $[s_{k_1},s_{k_2}] \sqcup [s_{k_3},s_{k_4}] \sqcup \dots \sqcup [s_{k_{2p-1}},s_{k_{2p}}]$.
    Notice that there are only finitely many choices for the set of indices $\{k_1, \dots, k_{2p}\}$.
    For any point $y \in H_0^n \cap [0,1]^n$ for which $\ell \cap X$ is non-empty, the value of $f(y)$ is then given by
    \[
        f(y) = \sum_{k=1}^p s_{j_{2k}} - s_{j_{2k-1}}.
    \]
    Thus, there are only finitely many possible values for $f$.
    This completes the proof of (\ref*{item:symm_is_cubical}).
    
    
    
    
    Now (\ref*{item:vol_pres}) follows from Fubini's theorem.  
    For (\ref*{item:relper_dec}), we need only prove that the inequality
    holds in each direction $e_j$.  
    That is, for each $e_j$, the relative perimeter of all components of
    $\partial X$ with normal vector $e_j$ is no larger than the corresponding relative perimeter of $\partial S_{e_i}(X)$
    in the direction $e_j$.  To compute this quantity for $X$, we integrate along the plane perpendicular to $e_i$ with
    an integrand $f(x)$, defined equal to the number of times the line in the direction $e_i$ through $x$ passes through the boundary
    of $X$.  If this number is infinite, we set $f(x)$ to be $0$.  Similarly, to compute this quantity for $S_{e_i}(X)$, we do the
    same, except we define $f(x)$ using the boundary of $S_{e_i}(X)$.
    
    Hence, the aforementioned relationship between the relative perimeters of $X$ and $S_{e_i}(X)$ is true 
    with equality if and only if every line $\ell$ with tangent vector $e_i$ which passes through the interior of $X$ intersects
    $\partial X$ in exactly two points, one of which is on $\partial [0,1]^n$.  This in turn is true if and only if $S_{e_i}(X) = X$ up to the isometries
    of $[0,1]^n$, thus proving (\ref*{item:relper_dec}) and (\ref*{item:relper_same}).

    Finally to prove (\ref*{item:succ_sym}), notice if $i=j$, then the statement holds trivially.
    So without loss generality, assume $i=1$ and $j=2$. 
    Suppose $S_{e_1}(X) = X$.
    Now consider an arbitrary two-dimensional plane $P$ spanned by vectors parallel to $e_1$ and $e_2$, i.e. $P=\{se_1 + te_2 + \sum_{k = 3}^n c_k e_k: s,t\in[0,1]\}$ for some constants $c_3,\dots,c_n\in[0,1]$.
    For a fixed value of $s\in[0,1]$ (resp. $t\in[0,1]$), let $\ell_s$ (resp. $\ell_t$) denote the line segment consisting of points in $P$ for which the $e_1$-coordinate (resp. $e_2$-coordinate) is equal to $s$ (resp. $t$).
    Because $S_{e_1}(X) = X$, the measure of $\ell_t \cap X$ is a non-increasing function of $t$.
    Thus, the measure of $\ell_t \cap S_{e_2}(X)$ is also a non-increasing function of $t$.
    Therefore, for each $s\in[0,1]$, the intersection $\ell_s \cap S_{e_2}(X)$ is either empty or is a connected line segment that intersects the coordinate hyperplane $H_0^1$, thus proving (\ref*{item:succ_sym}).
\end{proof}

Motivated by the previous proposition, we make the following:

\begin{defn}
    \label{defn:symmetrized}
    We say that a cubical subset $X \subset [0,1]^n$ is \emph{symmetrized} if
        $$S_{e_i} (X) = X$$
    for every $i \in \{1,\dots,n\}$.
\end{defn}

We can now define a \emph{special cubical subset}:
\begin{defn}
    \label{defn:special_cubical}
    A cubical subset $X \subset [0,1]^n$ of volume $V\in (0,1/2]$ is \emph{special} if three conditions are
    satisfied:
    \begin{enumerate}
        \item $X$ is symmetrized.
        \item There is an open cube $C \subset \mathbb{R}^n$ centered at $(0, \dots, 0)$
            such that $C \cap X = C \cap [0,1]^n$.
        \item For every $i \in \{1, \dots, n\}$, there is at most one singular point $s$ with respect to $i$ in $(0,1)$.
    \end{enumerate}
\end{defn}

For the next portion of the argument, we consider variations for an arbitrary symmetrized cubical subset $X$ which perturb a given singular slice in a normal direction, and we express the first variation of the relative perimeter for this perturbation.
To be more precise, consider $i \in \{1, \dots, n\}$, and suppose $s \in (0,1)$ is a singular point with respect to $e_i$ for $X$.
Let $S$ denote the slice $X_s^i$ of $X$ with respect to $s$ and $i$.
By the definition of a singular point, $S$ has positive $(n-1)$-dimensional Hausdorff measure; let this value be $A = \mathcal{H}^{n-1}(S) > 0$.
For $r \geq 0$, define the set $S(r)$ to be the Minkowski sum 
\[
    S(r) = S + \{ q e_i : q \in (-r,r) \}. 
\]
Now for $t\in(-\epsilon , \epsilon)$, where $\epsilon > 0$ is sufficiently small, we define the variation $X(t)$ of $X$ as follows:
\[
    X(t) = 
    \begin{cases}
        X \cup \overline{S(t/A)},        & t\geq 0,\\
        X \setminus S(|t|/A),   & t \leq 0,
    \end{cases}
\]
If $\epsilon$ is sufficiently small, then for all $t\in(-\epsilon,\epsilon)$, the set $X(t)$ is a cubical subset of $[0,1]^n$ with the same number of singular points as $X$ in any direction, and
\[
    \vol(X(t)) = \vol(X) + t.
\]


\begin{defn}
    Assume $X$ is a symmetrized cubical subset of $[0,1]^n$, $i \in \{1, \dots, n\}$, and $s \in (0,1)$ is a singular point for $X$ with respect to $e_i$.
    We refer to the \textit{first variation} of $X$ with respect to $s$ and $i$, denoted by $\firstvar(s,i)$, as the first variation of the relative perimeter with respect to the variation $X(t)$, as defined above.
    In particular, 
    \[
        \firstvar(s,i) = \tfrac{d}{dt} \left[\relper(X(t))\right] \big|_{t=0}.
    \]
\end{defn}

Our goal now is to express the first variation of the relative perimeter with respect to this variation $X(t)$ of $X$.
To accomplish this goal, consider $\partial S$. 
We will designate a sign to each $x \in \partial S$ denoted by $\psi(x)$; it will be a value in $\{-1, 0, 1\}$.  
Because $X$ is a symmetrized cubical subset of $[0,1]^n$, there are three possibilities:


\noindent \textbf{Possibility $\mathbf{0}$:} 
$x \in \partial [0,1]^n$, in which case, we set $\psi(x) = 0$.


\noindent {\bf Possibility $\mathbf{1}$: } 
$x \not\in \partial [0,1]^n$, and for every $\delta > 0$, the line segment 
\[
    \{ x + t e_i : t \in (s - \delta, s + \delta) \} 
\]
\emph{is not} entirely contained in $X$.  
In this case, we define $\psi(x)$ to be $1$.

\noindent {\bf Possibility $\boldsymbol{-1}$: } 
$x \not\in \partial [0,1]^n$, and there is some $\delta > 0$ so that the line segment
\[
    \{ x + t e_i : t \in (s - \delta, s + \delta) \} 
\]
\emph{is} entirely contained in $X$.  
In this case, we define $\psi(x)$ to be $-1$.

Since $S$ is a cubical subset of $H_s^i \cap [0,1]^n \cong [0,1]^{n-1}$, $\psi$ is integrable over $\partial S$ (with respect to the $(n-2)$-dimensional Hausdorff measure).  
As such, we may define 
\[
    P = \int_{\partial S} \psi(x) dx. 
\]

\begin{prop}
    \label{prop:firstvar}
    Assume $X$ is a symmetrized cubical subset of $[0,1]^n$, $i \in \{1, \dots, n\}$, and $s \in (0,1)$ is a singular point for $X$ with respect to $e_i$.
    Then 
    \[ 
        \firstvar(s,i) = \frac{P}{A} = \frac{\int_{\partial S} \psi(x) dx}{\mathcal{H}^{n-1}(S)},
    \]
    where $S$ is the slice $X_s^i$, and $\psi:\partial S \to \{-1,0,1\}$ is defined as above.
\end{prop}

\begin{proof}
    Let $X(t)$ be the variation of $X$ associated with $i$ and $s$ as defined above.
    Then for any $t \in (-\epsilon,\epsilon)$, where $\epsilon > 0$ is sufficiently small, we have 
    \begin{align*}
        \relper(X(t)) & =  \relper(X) + \tfrac{t}{A} [ 1 \cdot \mathcal{H}^{n-2}(\psi^{-1}(1))  + 0 \cdot \mathcal{H}^{n-2}(\psi^{-1}(0)) \\
        & \hspace{185pt} - 1 \cdot \mathcal{H}^{n-2}(\psi^{-1}(-1) ) ]\\
        & = \relper(X) + \frac{P}{A} t. \qedhere
    \end{align*}    
\end{proof}


There are two important observations involving the first variation.  The first involves two singular points
with respect to the same $i$:
\begin{lem}
    \label{lem:two_singular_same}
    Suppose that $X$ is a symmetrized cubical set of volume $V \in (0,1/2]$.
    Suppose that $0 < s_1 < s_2 < 1$
    are two singular points, both with respect to the same direction $e_i$ for some $i \in \{1, \dots, n\}$.  
    Then there is some
    symmetrized cubical set $\tilde{X}$ with the same volume as $X$ such that one of the following hold, depending on the values of $\firstvar(s_1,i)$ and $\firstvar(s_2,i)$ in $X$:
    \begin{enumerate}
        \item If $\firstvar(s_1,i) = \firstvar(s_2,i)$, then $\relper(\tilde{X}) = \relper(X)$, and the total number of singular points (in all directions) of $\tilde{X}$ is strictly
    less than the total number of singular points of $X$.\label{item:same}
        
        \item If $\firstvar(s_1,i) \neq \firstvar(s_2,i)$, then $\relper(\tilde{X}) < \relper(X).$\label{item:diff}
    \end{enumerate}
\end{lem}
\begin{proof}
    We consider the slices $S = X_{s_1}^i$ and $T = X_{s_2}^i$ corresponding to the singular points $s_1$ and $s_2$, respectively. 
    
    To prove (\ref*{item:same}), assume that $\firstvar(s_1,i) = \firstvar(s_2,i)$.   
    Suppose we perturb $X$ similarly to the variation $X(t)$ defined above, but where $S$ is moved in the direction $e_i$, $T$ is moved in the direction $-e_i$, and this movement is done simultaneously so that the volume of $X$ does not change.
    Then because the first variations at $s_1$ and $s_2$ are equal, the relative perimeter of $X$ does not change during this perturbation.
    To be more precise, if we perturb $S$ a distance of $\sigma$ in the direction $e_i$ and $T$ a distance of $\tau$ in the direction $-e_i$, then the total change in volume is equal to
        $$ \area(S) \sigma - \area(T) \tau $$
    and the change in relative perimeter is
        $$ \firstvar(s_1,i) \area(S) \sigma - \firstvar(s_2,i) \area(T) \tau.$$
    Thus, if we choose $\sigma > 0$ and $\tau > 0$ sufficiently small so that
        $$ \area(S) \sigma = \area(T) \tau, $$
    then the volume does not change, and since $\firstvar(s_1,i) = \firstvar(s_2,i)$, the relative perimeter also does not change.
    Hence, we can continue this perturbation until one of the following occurs:
    \begin{enumerate}
        \item $S$ and $T$ meet,
        \item $S$ meets another singular slice with respect to $i$, or
        \item $T$ meets another singular slice with respect to $i$.
    \end{enumerate}
    At this point the number of singular points in the direction $e_i$ has decreased, and the number of singular points
    in every other direction has not increased.  
    Indeed, suppose that $s$ is a singular point of the new (perturbed) set with respect to some $e_j$, where 
    $j \neq i$.  
    If $s$ was \emph{not} a singular point of the original set, then when we performed the perturbation, we could not have
    created a subset of the boundary with Hausdorff measure at least $n - 1$ in the hyperplane $H_s^j$.  This is because
    the boundary would intersect the hyperplane $H_s^j$ in a set of Hausdorff dimension \emph{at most} $n-2$.
    Thus, $s$ would have had to have been a singular point of the original set, too.  This completes this case.
    
    To prove (\ref*{item:diff}), assume that $\firstvar(s_1,i) < \firstvar(s_2,i)$.
    If we perturb $S$ so that the volume of $X$ increases by
    $\epsilon > 0$ sufficiently small, then the relative perimeter changes by $\epsilon \firstvar(s_1,i)$, and if we perturb $T$ so that the volume decreases by $-\epsilon$, then the relative perimeter changes by $-\epsilon \firstvar(s_2,i)$.
    Then during this perturbation, the total volume stays the same, but the relative perimeter decreases.
    If instead we have $\firstvar(s_1,i) > \firstvar(s_2,i)$, then we simply perturb $S$ so that the volume decreases by $-\epsilon$ and perturb $T$ so that the volume increases by $\epsilon$, and the result follows.
\end{proof}

We can now complete the proof of Proposition~\ref*{prop:reduction}:

\begin{proof}[Proof of Proposition~\ref{prop:reduction}]
    We begin by applying Steiner symmetrizations in the directions $e_1, \dots, e_n$ (in that order).  By Proposition~\ref{prop:steiner},
    the result has the same volume and has no larger relative perimeter.  If, for every $i$, there is at most one singular point in
    $(0,1)$, then we are done.
    
    If not, then there is some $i$ with $s_1 \neq s_2$ singular points in $(0,1)$.
    Thus, we can apply Lemma~\ref{lem:two_singular_same}
    to reduce the total number of singular points without reducing the volume and also without increasing the relative perimeter.
    We continue to apply this lemma until, for every $i$, there is at most one singular point in $(0,1)$ with respect to $i$.
    %
\end{proof}

Finally, we generalize Case (\ref*{item:diff}) of Lemma \ref{lem:two_singular_same} to compare singular points with respect to \textit{different} directions.
This lemma will be useful for identifying minimizers in Section \ref{sec:identify}.

\begin{lem}
    \label{lem:two_singular_different}
    Suppose that $X$ is a symmetrized cubical set of volume $V \in (0,1/2]$.  Suppose that
    $s_1$ is a singular point with respect to $i_1$, and $s_2$ is a singular point with respect to
    $i_2$, where $i_1 \neq i_2$.  
    If $\firstvar(s_1,i_1) \neq \firstvar(s_2,i_2)$,
    then there is some symmetrized cubical subset $\tilde{X} \subset [0,1]^n$ with the same volume
    as $X$, and $\relper(\tilde{X}) < \relper(X)$. 
\end{lem}
\begin{proof}
    The proof of this lemma works the same as the proof of Case \ref*{item:diff} in Lemma~\ref*{lem:two_singular_same}.
    The only difference is that the slices $S = X_{s_1}^{i_1}$ and $T = X_{s_2}^{i_2}$ lie in different hyperplanes and can potentially intersect, which cannot happen in Lemma~\ref*{lem:two_singular_same}.  
    
    In this case, we will consider two successive perturbations.
    First, we perturb $S$ in the direction $e_{i_1}$ by some distance $\sigma > 0$.
    For sufficiently small values of $\sigma$, this perturbation increases the volume of $X$ by $\area(S) \sigma$ and changes the relative perimeter by an amount of $\firstvar(s_1,i_1) \area(S) \sigma$.
    The resulting cubical region $X'$ has a singular slice $T'$ with respect to $i_2$ and $s_2$.
    Next, we perturb $T'$ in the direction $-e_{i_2}$ by some distance $\tau$.
    Then there exists a constant $C>0$ such that, for sufficiently small values of $\tau$, this perturbation decreases the volume of $X'$ by at most $-\area(T) \tau - C \tau^2$ and changes the relative perimeter by at most $-\firstvar(s_2,i_2) \area(T) \tau \pm C \tau^2$.

    Now choose $\sigma = \epsilon / \area(S)$ and $\tau = \epsilon / \area(T)$ for some sufficiently small $\epsilon > 0$.
    Then these two perturbations have changed the volume of $X$ in total by at most $C \epsilon^2$ and the relative perimeter by at most 
    \[
        | \firstvar(s_1,i_1) - \firstvar(s_2,i_2) | \epsilon + C \epsilon^2. 
    \]
    
    Thus, since $| \firstvar(s_1,i_1) - \firstvar(s_2,i_2) | > 0$, we can find such perturbations that
    maintain the volume but which reduces relative perimeter.  
    In other words, the perturbation decreases the relative perimeter on the first order with a nonzero coefficient,
    and hence (for small perturbations) the second order term does not matter for the purposes of this argument.
    
    In more detail, we choose $\nu > 0$ 
    and outward normal directions $n_1$ at $H_{s_1,i_1}$ and $n_2$ at $H_{s_2,i_2}$ so that the following.  If we perturb $H_{s_1,i_1}$ in the direction
    $n_1$ so that the the total volume increases by $\nu$, and the relative perimeter changes by $\firstvar(s_1,i_1) \nu$.
    Similarly, if we perturb $H_{s_2,i_2}$ in the
    direction $n_2$, the total volume increases by $\nu$, and the relative perimeter changes by $\firstvar(s_2,i_2) \nu$.  Thus, we can choose signs
    $\sigma_1, \sigma_2 \in \{-1, 1\}$ so that $$ \firstvar(s_1,i_1) \sigma_1 \nu - \firstvar(s_2,i_2) \sigma_2 \nu  < - A \nu$$ for some universal
    positive constant $A$ (as long as $\nu$ is sufficiently small).  This uses the fact that $| \firstvar(s_1,i_1) - \firstvar(s_2,i_2) | > 0$.
    
    If we perform both of these perturbations, then the total relative perimeter changes by at most  $-A \nu + B \nu^2$, while the total area changes by
    at most $D \nu^2$.  Here, $B$ and $D$ are universal positive constants (as long as $\nu$ is small enough).  Thus, if we modify the first perturbation by
    an amount on the order of $\nu^2$, we can correct this error in total volume, while still having a net change in relative perimeter of $-A \nu + B \nu^2 + E \nu^2$, where $E$ is again a universal positive constant (as long as $\nu$ is small enough).  As such, by choosing $\nu$ small enough we can obtain the 
    desired perturbation.
\end{proof}

\section{Identifying minimizers in dimension 3}\label{sec:identify}

To prove Theorem~\ref{thm:main}, we will need to show that certain configurations of special cubical subsets are minimizers for relative perimeter, while the remaining configurations cannot be minimizers.  

Consider the planes $H_\xi^i = \{ \sum x_j e_j \in \RR^3 : x_i = \xi \}$, where $i \in \{1,2,3\}$ and $\xi \in \{0,1\}$.
We begin with some basic observations about the faces $X \cap H_\xi^i$ of any special cubical subset $X \subset [0,1]^3$.  

\begin{lem}\label{lem:boundary}
    Suppose $X$ is a special cubical subset of $[0,1]^3$.
    For any $i$ and $\xi$, the face $X \cap H_\xi^i$ must take one of the following forms:
    \begin{enumerate}
        \item $\varnothing$, in which case $\xi = 1$, \label{item:boundempty}
        \item $[0,a] \times [0,b]$, with $0 < a , b \leq 1$, or \label{item:boundrect}
        \item $([0,a] \times [0,1]) \cup ([0,1] \times [0,b])$, with $0 < a , b < 1$. \label{item:boundL}
    \end{enumerate}
\end{lem}

\begin{proof}
    Suppose there exist $i$ and $\xi$ such that $X \cap H_\xi^i$ does not have one of the above forms.  
    Notice if $X \cap H_\xi^i$ in empty, then by the definition of special subset, $\xi = 1$; so assume $X \cap H_\xi^i$ is non-empty.
    Let $j,k\in\{1,2,3\}$ such that $i,j,k$ are mutually distinct.
    Then in at least one of the directions $e_j$ or $e_k$, there is more than one
    singular slice, contradicting the assumption that $X$ is special.  
    This is because we can do one of the following:
    
    Case 1: 
    We can find two distinct points $s_1$ and $s_2$ in $(0,1)$ so that the lines $\{ \xi e_i + s_1 e_j + t e_k : t \in [0,1] \}$
    and $\{ \xi e_i + s_2 e_j + t e_k : t \in [0,1] \}$ intersect $\partial X \cap H_\xi^i$ in subsets of Hausdorff dimension at least $1$.
    Then $s_1$ and $s_2$ are two singular points of $X$ with respect to $j$.
    
    Case 2:
    We can find two distinct points $s_1$ and $s_2$ in $(0,1)$ so that the lines $\{ \xi e_i + t e_j + s_1 e_k : t \in [0,1] \}$
    and $\{ \xi e_i + t e_j + s_2 e_k : t \in [0,1] \}$ intersect $\partial X \cap H_\xi^i$ in subsets of Hausdorff dimension at least $1$.
    Then $s_1$ and $s_2$ are two singular points of $X$ with respect to $k$, which is a contradiction.
\end{proof}

Next, we identify relative perimeter minimizers.

\begin{prop}
    Suppose $X$ is a special cubical subset of $[0,1]^3$.
    If $X$ minimizes relative perimeter and the faces $X \cap H_\xi^i$ are of the form (\ref*{item:boundempty}) or (\ref*{item:boundrect}) from Lemma \ref*{lem:boundary} for all $i$ and $\xi$, then $X$ must be of the form described in the conclusion of Theorem~\ref{thm:main}, meaning $X = [0,a]^m \times [0,1]^{3-m}$ for some $a\in(0,1)$ and $m\in\{1,2,3\}$.
\end{prop}

\begin{proof}
    Under these hypotheses, $X$ must be of one of the following types:
    
    \begin{enumerate}[(i)]
        \item $[0,a] \times [0,b] \times [0,c]$ with $0 < a,b,c < 1$. \label{item:abcprism}
        \item $[0,a] \times [0,b] \times [0,1]$ with $0 < a,b < 1$.\label{item:abtube}
        \item $[0,a] \times [0,1]^2$ with $0 < a \leq 1$.\label{item:aslab}
    \end{enumerate}
    
    If $X$ is to be a relative perimeter minimizer, then the first variations of any singular slices must be equal by Lemma~\ref{lem:two_singular_different}.
    In case (\ref*{item:abcprism}), setting the first variations of the three singular slices equal gives
    \[
       \frac{a + b}{ab} =  \frac{b + c}{bc} = \frac{a + c}{ac} \implies a = b = c.
    \]
    Thus $X = [0,a]^3$, as desired.
    In case (\ref*{item:abtube}), we have 
    \[
        \frac{1}{a} = \frac{1}{b} \implies a = b.
    \]
    Thus $X = [0,a]^2 \times [0,1]$, as desired.
    Finally for case (\ref*{item:aslab}), $X$ is already of a desired form.
\end{proof}

%
%
%
%

Finally to complete the proof of Theorem \ref{thm:main}, we show that the remaining special subsets cannot be relative perimeter minimizers.

\begin{prop}
    Suppose $X$ is a special cubical subset of $[0,1]^3$.
    If there exists $i$ and $\xi$ such that the face $X \cap H_{\xi}^{i}$ is of the form (\ref*{item:boundL}) from Lemma \ref*{lem:boundary}, then $X$ is not a minimizer for relative perimeter.
\end{prop}

\begin{proof}
    Assume $X$ satisfies the hypotheses above.
    We have the following cases to consider for the faces $X \cap H_\xi^i$ where $\xi = 0$ and $i\in\{1,2,3\}$:
    \begin{enumerate}[(a)]
        \item All three faces $X \cap H_0^1$, $X \cap H_0^2$, $X \cap H_0^3$ are of the form (\ref*{item:boundrect}).\label{item:03}
        \item One of the faces is (\ref*{item:boundL}), while the other two are (\ref*{item:boundrect}).\label{item:12}
        \item Two of the faces are of the form (\ref*{item:boundL}), while the third is (\ref*{item:boundrect}).\label{item:21}
        \item All three faces  are of the form (\ref*{item:boundL}).\label{item:30}
    \end{enumerate}
    Each of these cases is illustrated in Figure~\ref*{fig:candidates}.

    \begin{figure}[h]
        \includegraphics[width=\textwidth]{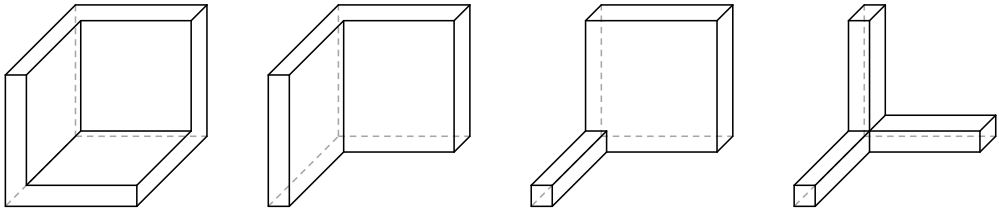}
        \caption{Special cubical subsets that are not minimizers of relative perimeter.}
        \label{fig:candidates}
    \end{figure}
    
    In Case (\ref*{item:03}), by assumption, at least one of the faces $X \cap H_1^1$, $X \cap H_1^2$, $X \cap H_1^3$ is of the form (\ref*{item:boundL}).
    It then follows that,
    \[
        X = ([0,a] \times [0,1]^2) \cup ([0,1] \times [0,b] \times [0,1]) \cup ([0,1]^2 \times [0,c])
    \]
    for some $a,b,c\in(0,1)$.
    If $X$ is to be a relative perimeter minimizer, then the first variations of any singular slices must be equal by Lemma~\ref{lem:two_singular_different}.
    Hence, calculating the first variation of the three singular slices, we have:
    \[
        \frac{a + b - 2}{(1 - a)(1 - b)} = \frac{a + c - 2}{(1 - a)(1 - c)} = \frac{b + c - 2}{(1 - b)(1 - c)}.
    \]
    Thus, we obtain $a = b = c$.  
    Since the volume must be $\leq 1/2$, 
    \[
        2 a^3 - 6a^2 + 6a - 1 \leq 0.
    \]
    From this, we conclude that $0 \leq a \leq 3/10$. 
    If we delete two of the faces and replace it by a single slab, the area goes down
    by $2(1-a)^2$ and goes up by $2a - a^2$, so the total change is
    \[
        2a - a^2 - 2(1-a)^2 = -3a^2 +6a - 2.
    \]
    On $[0,3/10]$, this quadratic is negative, and so doing this replacement reduces the total relative perimeter while maintaining
    volume.
    Thus, $X$ cannot be a minimizer in this Case (\ref*{item:03}).
        
    In Case (\ref*{item:12}), it follows that up to isometries of $[0,1]^3$,
    \[
        X =  \big( ([0,a] \times [0,1]) \cup ([0,1] \times [0,b]) \big) \times [0,c]
    \]
    for some $a,b\in(0,1)$ and $c \in (0,1]$.
    Let $Y = ([0,a] \times [0,1]) \cup ([0,1] \times [0,b])$, and let $\vol(Y)$ be the $2$-dimensional volume of $Y$,
    and $\relper(Y)$ be the $1$-dimensional relative perimeter of $Y$ as a subset of $[0,1]^2$.  If $c < 1$,
    then
        $$ \relper(X) = c \relper(Y) + \vol(Y), $$
    and if $c = 1$, then
        $$ \relper(X) = c \relper(Y). $$
    Furthermore, in both cases,
        $$ \vol(X) = c \vol(Y). $$
    As discussed in the introduction, we have a complete characterization for this problem in $2$-dimensions.  If we replace $Y$ with
    $Y'$ so that $\relper(Y') < \relper(Y)$ and $\vol(Y') = \vol(Y)$, then keeping $c$ the same,
    we will have produced a better competitor for $X$ (regardless of whether $c = 1$ or $c < 1$).
    Since the minimizers for $Y$ are squares of the form $[0,a]^2$ or slabs of the form $[0,a] \times [0,1]$ up to
    isometries of $[0,1]^2$, since $Y$ in this case is not of one of these forms, we can indeed find a better competitor,
    completing this case.
    
    In Case (\ref*{item:21}), it follows that up to isometries of $[0,1]^3$, 
    \[
        X = ([0,a] \times [0,1]^2) \cup ([0,1] \times [0,b] \times [0,c])
    \] 
    for some $a,b,c\in(0,1)$.
    Equating the first variations of the three singular slices gives
    \[
        \frac{-b - c}{1 - bc}=\frac{1 - a - c}{(1 - a)c}=\frac{1 - a - b}{(1 - a)b}.
    \]
    By the second equality, $b = c$.  
    Since the total volume must be $\leq 1/2$, $a \leq 1/2$.  
    Replacing the $b$-$c$
    leg with just a slab of the correct volume, the relative area increases by $b^2$ but decreases by $2b(1 - a) > b$.
    Since $b < 1$, $b^2 < b$, and so the total relative perimeter goes down.
    
    As an alternative argument, since these quantities are all equal,
    \[
        \frac{-2b}{1-b^2} = \frac{1-a-b}{(1-a)b},
    \]
    $0 \leq b \leq 1$, and so we can show from this equality that $a > 1/2$, which is a contradiction.
    
    In Case (\ref*{item:30}), it follows that up to isometries of $[0,1]^3$, 
    \[
        X = ([0,a] \times [0,1] \times [0,c]) \cup ([0,1] \times [0,b] \times [0,c]) \cup ([0,a] \times [0,b] \times [0,1])
    \]
    for some $a,b,c\in(0,1)$.
    If $X$ is to be a relative perimeter minimizer, then the first variations of any singular slices must be equal by Lemma~\ref{lem:two_singular_different}.
    Hence
    \[
        \frac{1 - a - c}{(1 - a)c} = \frac{1 - a - b}{(1 - a)b} = \frac{1 - b - a}{(1 - b)a} = \frac{1 - b - c}{(1 - b)c} = \frac{1 - c - a}{(1 - c)a} = \frac{1 - c - b}{(1 - c)b}.
    \]
    It follows that $a = b = c$.  Since the total volume must be at most $1/2$, $a \leq 1/2$.  Thus, if we cut off one of the three ``legs,"
    that is, remove $[0,a]^2\times[a,1]$,
    rotate it $\pi/2$ radians, and then glue it to the other face, that is, union with $[a,2a]\times[a,1]\times[0,a]$, the
    volume is preserved and the perimeter is reduced from $6a(1-a)$ to $5a(1-a)$.
    Thus $X$ is not a minimizer in this case.
\end{proof}

\section{Uniqueness of minimizers in dimension 3}\label{sec:unique}

In this section, we discuss the uniqueness of the minimizers we have identified in the previous section.  
First, we note that any minimizer can be modified via
one of the isometries of $[0,1]^3$.  
Thus, all of our statements about uniqueness will be up to actions by isometry group of $[0,1]^3$.

Next, there are volumes $0 < V_1 < V_2 < 1$ such that the relative perimeter of a cube $[0,a]^3$ of volume $V_1$ is equal to the relative perimeter of a
tube $[0,b]^2 \times [0,1]$ of volume $V_1$, and the relative volume of a tube of volume $V_2$ is equal to that of a slab $[0,c]\times[0,1]^2$ of volume $V_2$.  
We calculate $V_1$ and $V_2$ here.
For the volume $V_1$, the relative perimeter of the cube is $3 V_1^{2/3}$, and the relative perimeter of the tube is $2 V_1^{1/2}$. Thus, $V_1 = ( \frac{2}{3} )^6$.
For $V_2$, the relative perimeter of the tube is $2 V_2^{1/2}$, and that of the slab is $1$.  Therefore, $V_2 = \frac{1}{4}$.

We can now state our main uniqueness result:
\begin{prop}
    \label{prop:uniqueness}
    Define $V_1$ and $V_2$ as above.
    Considering cubical subsets of $[0,1]^3$ of volume $V \in (0,1/2]$ that are minimizers of relative perimeter, the following are true up to isometries of $[0,1]^3$ or sets of measure zero:
    \begin{enumerate}
        \item If $V < V_1$, then the only minimizer is the cube of volume $V$.
        \item If $V = V_1$, then the only minimizers are the cube and the tube of volume $V_1$.
        \item If $V_1 < V < V_2$, then the only minimizer is the tube of volume $V$.
        \item If $V = V_2$, then the only minimizers are the tube and the slab of volume $V_2$.
        \item If $V > V_2$, then the only minimizer is the slab of volume $V$.
    \end{enumerate}
    Here, the cube of volume $V$ refers to $\big[0,V^{1/3}\big]^3$, the tube of volume $V$ is $\big[0,V^{1/2}\big]^2 \times [0,1]$, and the slab of volume $V$ is $[0,V] \times [0,1]^2$.
\end{prop}

\begin{proof}
    We begin with a cubical set $X$ of volume $V \in (0,1/2]$ which is 
    a minimizer of relative perimeter. 
    By Proposition \ref{prop:steiner}, we may assume that $X$ has been Steiner symmetrized in the three coordinate directions.
    
    Now if $X$ is special, then $X$ must be a cube, tube, or slab by our arguments in Section \ref{sec:identify}.
    Otherwise, $X$ is not special, and because its volume is positive, it follows that there must exist a direction $e_i$ with respect to which $X$ has two or more singular points, meaning $X$ is not a cube, tube, or slab.
    By Lemma \ref{lem:two_singular_same}, given a direction $i$ that has multiple singular slices, these slices must have the same first variation.
    Then by iteratively applying the variation described in the proof of Lemma \ref{lem:two_singular_same} on the singular slices of $X$, the number of singular slices can be reduced until the resulting special cubical region $X^*$ has the same volume and relative perimeter as $X$ and is either a cube, tube, or slab by the results in Section \ref{sec:identify}.
    Consider the last iteration of this process, taking a symmetrized cubical subset $\tilde{X}$ and perturbing it until we arrive at $X^*$.
    $\tilde{X}$ has the same volume and relative perimeter as $X^*$, it has a direction (without loss of generality, say it is $e_1$) with respect to which there is exactly two singular slices, $S = \tilde{X}_{s_1}^1$ and $T = \tilde{X}_{s_2}^1$, corresponding to singular points $s_1 < s_2$ in the direction $e_1$.
    These slices have the same first variation, and the perturbation moves $S$ in the direction $e_1$ and $T$ n the direction $-e_1$ until they lie in the same plane.
    We consider the three cases: $X^*$ is a cube, tube, or slab.
    
    
    Suppose $X^*$ is a cube $[0,a]^3$.  
    Considering $T$ as a subset of the face $A = \{a\} \times [0,a]^2$, we will calculate first variations.  
    The first variation of the portion $T$ is
    $$ \frac{\relper(T)}{\area(T)}, $$
    and the first variation of $S$ is
    $$ \frac{2a - \relper(T)}{a^2 - \area(T)}. $$
    Because the first variations must be equal, it follows that
    \begin{equation}\label{eq:cube1stvar}
        \frac{\relper(T)}{\area(T)} = \frac{2}{a}.
    \end{equation}
    On the other hand, we observe that $\relper(T) \geq 2 \sqrt{\area(T)}$ from the $2$-dimensional cubical relative isoperimetric problem which we discussed
    in the introduction.  
    In addition, since $T$ is a strict subset of $[0,a]^2$, $\area(T) < a^2$.  
    Thus, it follows that
    \[
        \frac{\relper(T)}{\area(T)} \geq \frac{2}{\sqrt{\area(T)}} > \frac{2}{a},
    \]
    which contradicts (\ref*{eq:cube1stvar}).
    This completes the case of the cube.
    
    The next case is a tube $[0,a]^2 \times [0,1]$.
    Again, consider $T$ as a subset of the face $A = \{a\} \times [0,a] \times [0,1]$
    As before, the first variation of the $T$ is
    $$ \frac{\relper(T)}{\area(T)}, $$
    and the first variation of $S$ is
    $$ \frac{1 - \relper(T)} {a - \area(T)}. $$
    Again, these must be equal, and so we have
    \begin{equation}\label{eq:tube1stvar}
        \frac{\relper(T)}{\area(T)} = \frac{1}{a}.
    \end{equation}
    Since $T$ is a proper subset of the face, $\area(T) < a$.  
    In addition, from the solution to the $2$-dimensional problem for the face
    $A = \{a\} \times [0,a] \times [0,1]$ (see Remark \ref{rem:strip} below), we have the following:  
    If $\area(T) \leq a^2$, then $\relper(T) \geq 2\sqrt{\area(T)}$.
    Then
        $$ \frac{\relper(T)}{\area(T)} \geq \frac{2}{\sqrt{\area(T)}} > \frac{1}{a} $$
    which is a contradiction.  
    If $\area(T) > a^2$, then $\relper(T) \geq \min\{1,a + \frac{\area(T)}{a}\}$.  
    If $\relper(T)\geq 1$, then we have 
        $$ \frac{\relper(T)}{\area(T)} \geq \frac{1}{\area(T)} \geq \frac{1}{a}.  $$
    Since these all must be equalities by (\ref*{eq:tube1stvar}), we have $\area(T) = a$, which contradicts the fact that $T$ is a proper subset of $A$.  
    If instead $\relper(T) \geq a + \frac{\area(T)}{a}$, we have
        $$ \frac{\relper(T)}{\area(T)} \geq \frac{a}{\area(T)} + \frac{1}{a} > \frac{1}{a}.$$
    However, this contradicts (\ref*{eq:tube1stvar}).
    This completes the tube component of the proof.

    Finally, the last case is a slab $[0,a] \times [0,1]^2$.
    Consider $T$ as a proper subset of the face $A = \{a\} \times [0,1]^2$.  
    The first variation of the component corresponding to $T$ is then positive, but the first variation of $S$ is negative.  
    Since they have to be equal, we arrive at a contradiction, thus completing the proof.
\end{proof}

\begin{rem}\label{rem:strip}
    We can extend the technique discussed in Remark \ref{rem:dim2} to resolve the relative isoperimetric problem for cubical subsets of $[0,1]^2$ while restricting the sets to lie in $[0,a] \times [0,1]$ for some fixed $a \in(0,1)$. 
    Here, if the area $V$ of the minimizer is at most $a^2$, then the minimizer (up to isometries) is $\big[0,V^{1/2}\big]^2$.  
    If the area $V$ is greater than $a^2$, then the minimizer is either a strip of the form $[0,V] \times [0,1]$
    or a rectangle of the form $[0,a] \times [0,\frac{V}{a}]$.  
    In particular, this implies that the minimal relative perimeter is $\min\left( 1, a + \frac{V}{a} \right)$.
\end{rem}

\bibliographystyle{amsplain}
\bibliography{bibliography}

\providecommand{\bysame}{\leavevmode\hbox to3em{\hrulefill}\thinspace}
\providecommand{\MR}{\relax\ifhmode\unskip\space\fi MR }
\providecommand{\MRhref}[2]{%
  \href{http://www.ams.org/mathscinet-getitem?mr=#1}{#2}
}
\providecommand{\href}[2]{#2}
\begin{thebibliography}{10}

\bibitem{BLL}
H.~J. Brascamp, E.~H. Lieb, and J.~M. Luttinger, \emph{A general rearrangement
  inequality for multiple integrals}, J. Funct. Anal. \textbf{17} (1974),
  227--237.

\bibitem{BurchardCourse}
Almut Burchard, \emph{{A Short Course on Rearrangement Inequalities}}, 2009,
  unpublished lecture notes.

\bibitem{Lawson1970}
{H. Blaine} {Lawson, Jr.}, \emph{Complete minimal surfaces in $\mathbb{R}^3$},
  Ann. of Math. \textbf{92} (1970), no.~3, 335--374.

\bibitem{Ledoux}
Michel Ledoux, \emph{Concentration of measure and logarithmic {Sobolev}
  inequalities}, S{\'e}minaire de Probabilit{\'e}s XXXIII (Berlin, Heidelberg)
  (Jacques Az{\'e}ma, Michel {\'E}mery, Michel Ledoux, and Marc Yor, eds.),
  Lecture Notes in Mathematics, vol. 1709, Springer Berlin Heidelberg, 1999,
  pp.~120--216.

\bibitem{LM}
L.~H. Loomis and H.~Whitney, \emph{An inequality related to the isoperimetric
  inequality}, Bull. Amer. Soc. \textbf{55} (1949), no.~10, 961 -- 962.

\bibitem{Wulff}
Frank Morgan, Christopher French, and Scott Greenleaf, \emph{Wulff clusters in
  $\mathbb{R}^2$}, J. Geom. Anal. \textbf{8} (1998), 97 -- 115.

\bibitem{Rit}
Manuel Ritor\'{e}, \emph{Superficies con curvatura media constante}, Ph.D.
  thesis, Universidad de Granada, 1994.

\bibitem{RitRos}
Manuel Ritor\'{e} and Antonio Ros, \emph{The spaces of index one minimal
  surfaces and stable constant mean curvature surfaces embedded in flat three
  manifolds}, Trans. Amer. Math. Soc. \textbf{348} (1996), no.~1, 391 -- 410.

\bibitem{Ros2020THEIP}
Antonio Ros, \emph{The isoperimetric problem}, Proceedings of the Clay
  Mathematical Institute MSRI summer school on Minimal Surfaces (2005).

\bibitem{Ross1992}
Marty Ross, \emph{Schwarz' {P} and {D} surfaces are stable}, Differential Geom.
  Appl. \textbf{2} (1992), 179--195.

\bibitem{steiner}
Jakob Steiner, \emph{Einfacher {B}eweis der isoperimetrischen {H}aupts\"atze},
  J. Reine Angew. Math. \textbf{18} (1838), 281--296.

\bibitem{crystal}
Brian Wecht, Megan Barber, and Jennifer Tice, \emph{Double crystals}, Acta
  Crystallogr. Sect. A \textbf{A56} (2000), no.~1, 92 -- 95.

\end{thebibliography}

\end{document}